\documentclass[12pt]{article}
\usepackage{amssymb,amsmath}
\usepackage{cancel}
\usepackage{palatino}
\usepackage{longtable}
\usepackage{dsfont}
\usepackage{amsthm}
\usepackage[normalem]{ulem}
\usepackage[hmargin=3cm,vmargin=3cm]{geometry}
\usepackage{color}
\usepackage{mathtools}
\numberwithin{equation}{section}
\theoremstyle{definition}
\newtheorem{theorem}{Theorem}[section]   
\newtheorem{thm}{Theorem}[section]     
\newtheorem{definition}[theorem]{Definition}
\newtheorem{deft}[theorem]{Definition}
\newtheorem{proposition}[theorem]{Proposition}

\newtheorem{prop}[theorem]{Proposition}
\newtheorem{lemma}[theorem]{Lemma}

\newtheorem{cor}[theorem]{Corollary}

\newtheorem{example}[theorem]{Example}

\newtheorem{example-notation}[theorem]{Example-Notation}

\newtheorem{rem}{Remark}

\newcommand{\eqa}{\begin{eqnarray}}
\newcommand{\eeqa}{\end{eqnarray}}
\newcommand{\beq}{\begin{equation}}
\newcommand{\eeq}{\end{equation}}

\allowdisplaybreaks

\begin{document}
\title{Regular $F$-manifolds with eventual identities}
\author{Sara Perletti and Ian A.B. Strachan}

%\date{12$^{\rm th}$ Jan 2024}

\date\today

\maketitle\thispagestyle{empty}

\vspace{-0.2in}

\begin{abstract} 
Given an $F$-manifold one may construct a dual multiplication (generalizing the idea of an almost-dual Frobenius manifold introduced by Dubrovin) using a so-called eventual identity, the definition of which ensure that the dual object is also an $F$-manifold. In this paper we solve the equations for an eventual identity for a regular (so non-semi-simple) $F$-manifold and construct a dual coordinate system in which dual multiplication is preserved. As an application, families of Nijenhuis operators are constructed.
\end{abstract}    
\addtocontents{toc}{\protect}
\tableofcontents

\section{Introduction}
\subsection{$F$-manifolds}
The definition of an $F$-manifold was given by Hertling and Manin twenty-five years ago \cite{HM99}. Originating in the the theory of Frobenius manifolds and singularity theory, it has proved to be an extremely versatile concept, finding applications in, for example, the theory of integrable systems and even information theory \cite{CombeManin}.

\begin{definition}\label{defFmani}
	An \emph{$F$-manifold} is a triple $(M,\circ,e)$ consisting of a manifold $M$, a commutative and associative multiplication $\circ$ on the tangent bundle with the properties:
\begin{itemize}
\item[(i)] for all vector fields $X,Y,W, Z\in\mathfrak{X}(M)$,
	\begin{equation}
		\mathcal{L}_{X\circ Y}(\circ)=X\circ\mathcal{L}_Y(\circ)+Y\circ\mathcal{L}_X(\circ)\,,
		\label{HMeq2}
	\end{equation}	
where $\mathcal{L}_XY = [X,Y]$ is the Lie bracket of vector fields;
\item[(ii)] there exists a distinguished vector field $e$ on $M$, the unit of $\circ$, such that ${e\circ X=X}$ for all vector fields $X\,.$
 \end{itemize}
\end{definition}

\noindent We we also require the notion of an Euler vector field:

\begin{deft}
	An \emph{Euler vector field} on an $F$-manifold $(M,\circ,e)$ is a vector field $E$ satisfying the condition
	\begin{equation}
		\mathcal{L}_E(\circ)(X,Y)=X\circ Y,\qquad X,Y\in\mathfrak{X}(M).
	\end{equation}
	This means that $E$ preserves the multiplication up to a constant.
\end{deft}
\noindent While many properties of $F$-manifolds have been derived (see \cite{He02}), the classification of $F$-manifolds remains an open problem.

\medskip

Motivated by Dubrovin's construction of \lq almost-dual\rq~Frobenius manifolds \cite{Dub04}, Manin introduced a new commutative and associative multiplication:
	\begin{equation}
		X*Y=\mathcal{E}^{-1}\circ X\circ Y,\qquad X,Y\in\mathfrak{X}(M),
		\label{dualmultiplic}
	\end{equation}
for an arbitrary invertible vector field $\mathcal{E}$ (invertible meaning that the vector field $\mathcal{E}^{-1}$ satisfies the condition $\mathcal{E}\circ \mathcal{E}^{-1}=e$). In general, $\mathcal{E}$ will not be defined everywhere on the manifold - there will be some submanifold $\Sigma$ where $\mathcal{E}$ is not invertible. For simplicity, we will just refer to the manifold as $M$ rather than separately to the manifolds $M$ and $M\backslash \Sigma$. Clearly $\mathcal{E}$ is a unit for $*$. Manin then defined an eventual identity as a vector field $\mathcal{E}$ that preserves the $F$-manifold structure \cite{Ma05}.
\begin{deft}
	An \emph{eventual identity} for an $F$-manifold $(M,\circ,e)$ is an invertible vector field $\mathcal{E}$ such that the dual multiplication (\ref{dualmultiplic}) defines an $F$-manifold structure $(M,*,\mathcal{E})$.
\end{deft}
\noindent The characterization of such eventual identities was given in \cite{DS}:
\begin{thm}\label{DStheorem}
Given an $F$-manifold $(M,\circ,e)$, an invertible vector field $\mathcal{E}$ is an eventual identity if and only if
	\begin{equation}
		\mathcal{L}_\mathcal{E}(\circ)(X,Y)=[e,\mathcal{E}]\circ X\circ Y,\qquad X\text{,}Y\in\mathfrak{X}(M)\text{.}
		\label{Lie}
	\end{equation}
\end{thm}

\medskip

\noindent The multiplication $\circ$ is said to be semi-simple if, at a generic point in the manifold, there exists a set of idempotent vector fields $\{\delta_i\}_{i=1,\dots,n}$ with the property
\[
\delta_i \circ \delta_j = \delta_{ij} \delta_i,\qquad i,j=1,\dots,n,
\]
i.e. the multiplication decomposes into one-dimensional blocks. It may be shown that canonical coordinates $\{u^i\}_{i=1,\dots,n}$ exists in which the idempotents are coordinate vector fields $\delta_i=\partial_i\,$, ${i=1,\dots,n}$. In such a coordinate system it is straightforward to construct eventual identities by solving the conditions given in the above theorem (see \cite{DS}), giving
\[
\mathcal{E}=\overset{n}{\underset{i=1}{\sum}}\mathcal{E}^i(u^i)\frac{\partial}{\partial u^i}\text{.}
\]
Thus, if the multiplication is semi-simple, eventual identities are defined by $n$-functions of one-variable. We will often denote $\frac{\partial}{\partial u^i}=\partial_{u^i}$, $i=1,\dots,n$.

\medskip

Without semi-simplicity, it is much harder to study $F$-manifolds: indeed, there is no classification of $F$-manifolds beyond two dimensions, with only a partial classification in three dimensions \cite{BH}. In \cite{DH16} a so-called regular-$F$-manifold was defined.

\begin{deft}
An $F$-manifold $(M,\circ,e,E)$ with Euler field is called \emph{regular} at a point $m\in M$ if the endomorphism $E\circ|_m:T_mM\to T_mM$ is regular, namely each of its Jordan blocks is associated to a different eigenvalue. The $F$-manifold is \emph{(generically) regular} if it is regular at any (generic) point.
\end{deft}
\noindent Furthermore, an extension of the idea of canonical coordinates was given. Let $(M,\circ,e,E)$ be a regular $F$-manifold of dimension $n$ with the operator $E\circ$ having $r$-Jordan blocks each of size $m_1,\dots,m_r$ . In (generalized) canonical coordinates
	$$\{u^{j(\alpha)}\,|\,\alpha=1,\dots,r,\,j=1,\dots,m_\alpha\}$$
	the structure constants of the product read
	\begin{equation}
		c^{i(\alpha)}_{j(\beta)k(\gamma)}=\delta^\alpha_\beta\delta^\alpha_\gamma\delta^i_{j+k-1}
		\notag
	\end{equation}
	and the components of the unit vector field are $e^{i(\alpha)}=\delta^i_1\,,$ where
	\[
	j(\alpha)=j+(1-\delta^\alpha_1)\overset{\alpha-1}{\underset{\sigma=1}{\sum}}m_\sigma\text{,}
	\]	
	for all suitable indices.

\subsection{The structure of the paper}

The main aim of this paper is to solve the eventual identity equations \eqref{Lie} for regular $F$-manifolds. It turns out that the structure of such fields is considerable more complicated than in the semi-simple case. Since a regular $F$-manifold consists over several Jordan blocks, we first, in Section 2, solve the equations for an eventual identity $\mathcal{E}$ for a single Jordan block. We then examine the dual structure defined via the multiplication \eqref{dualmultiplic}, constructing a new basis of vector fields with respect to which the dual product preserves the structure of the original product on the $F$-manifold in terms of the Jordan block arrangement of the operator $\mathcal{E}\circ$\,, and in Section 3 we extend these considerations to the general case of multiple Jordan blocks.

\section{Eventual identities and dual coordinates for a single Jordan block}
We first solve the equations for an eventual identity for a single Jordan block of size $n$, where $r=1$, dropping the Greek indices for ease of notation. After this, we construct - via Frobenius' theorem - a dual coordinate system in which the dual multiplication takes the same form as the original multiplication. This is a subtle point that does not exists in the semi-simple case, and involves placing a constraint on an otherwise-free function.

\subsection{Eventual identities for a single Jordan block}
\begin{example}
	Let us consider the case of one Jordan block in dimension $n=3$. Condition \eqref{Lie} gives
	\begin{equation}
		\mathcal{L}_\mathcal{E}(\circ)(\partial_i,\partial_j)=[e,\mathcal{E}]\circ \partial_i\circ \partial_j
		\label{Lie_ex3}
	\end{equation}
	for $i,j=1,2,3$, where
	\begin{align}
		\partial_1\circ\partial_i&=\partial_i,\qquad i=1,2,3,\notag\\
		\partial_2\circ\partial_2&=\partial_3,\notag\\
		\partial_2\circ\partial_3&=\partial_3\circ\partial_3=0\notag.
	\end{align}
	Condition \eqref{Lie_ex3} holds trivially when $i=1$ or $j=1$. For $i=j=\partial_3$, it amounts to
	\begin{equation}
		\partial_3\mathcal{E}^1=0,\label{Ex3_I}
	\end{equation}
	while for $i=2$ and $j=3$ (equivalently, $i=3$ and $j=2$) it amounts to
	\begin{equation}
		\partial_3\mathcal{E}^2=-\partial_2\mathcal{E}^1.\label{Ex3_II}
	\end{equation}
	For $i=j=2$, condition \eqref{Lie_ex3} gives
	\begin{equation}
		\partial_3\mathcal{E}^2=2\,\partial_2\mathcal{E}^1\label{Ex3_III}
	\end{equation}
	and
	\begin{equation}
		\partial_3\mathcal{E}^3=2\,\partial_2\mathcal{E}^2-\partial_1\mathcal{E}^1.\label{Ex3_IV}
	\end{equation}
	By comparing \eqref{Ex3_II} with \eqref{Ex3_III}, we get
	\begin{equation}
		\partial_3\mathcal{E}^2=\partial_2\mathcal{E}^1=0,\label{Ex3_V}
	\end{equation}
	which together with \eqref{Ex3_I} imposes the requirement for $\mathcal{E}^1$ to be a function of the sole variable $u^1$ and for $\mathcal{E}^2$ to be a function of the sole variables $u^1$, $u^2$. Condition \eqref{Ex3_IV} demands the function $\mathcal{E}^3$ to be of the form
	\begin{equation}
		\mathcal{E}^3=\big(2\,\partial_2\mathcal{E}^2-\partial_1\mathcal{E}^1\big)u^3+f_3(u^1,u^2)
		\notag
	\end{equation}
	for some function $f_3$ of $u^1$, $u^2$.
\end{example}
The main result in this section is:
\begin{proposition}\label{formofE}
	An invertible vector field $\mathcal{E}\in\mathfrak{X}(M)$ is an eventual identity if and only if
	\begin{equation}
		\partial_l\mathcal{E}^m=\begin{cases}
			(l-1)\,\partial_2\mathcal{E}^{m-l+2}-(l-2)\,\partial_1\mathcal{E}^{m-l+1}\qquad&\text{for }l\leq m\\
			0\qquad&\text{for }l>m
		\end{cases}
	\label{Atlas}
	\end{equation}
	for each $m=1,\dots,n$.
\end{proposition}
\begin{proof}
	Consider \eqref{Lie} with $X=\partial_j$ and $Y=\partial_k$. This gives
		\begin{align}
		&\mathcal{L}_\mathcal{E}(\partial_j\circ\partial_k)-[\mathcal{E},\partial_j]\circ\partial_k-[\mathcal{E},\partial_k]\circ\partial_j=[e,\mathcal{E}]\circ \partial_j\circ \partial_k
		\notag
	\end{align}
		which simplifies, for each $i=1,\dots,n$, to
	\begin{align}
		&-\partial_{j+k-1}\mathcal{E}^i+\partial_j\mathcal{E}^{i-k+1}+\partial_k\mathcal{E}^{i-j+1}=\partial_1\mathcal{E}^{i-j-k+2}\text{.}
		\label{AlmostAtlas}
	\end{align}
	For notational simplicity we do not specify the range of the indices - assuming that $\mathcal{E}^i=0$ if the label falls outside the range $1\,,\ldots\,,n\,.$ Condition \eqref{AlmostAtlas} immediately follows from \eqref{Atlas}. We now show that \eqref{AlmostAtlas} implies \eqref{Atlas}. Let us fix $m\leq n$. First, we show that $\partial_l\mathcal{E}^m=0$ for each $l>m$. In particular, we first deal with the case $l\geq m+2$ and then with $l=m+1$. By choosing $k=2$ and $j\geq i+1$ in \eqref{AlmostAtlas}, we get
	\begin{align}
		&\partial_{j+1}\mathcal{E}^i=\partial_j\mathcal{E}^{i-1},\qquad j\geq i+1\text{.}
		\label{AlmostAtlas_reduced}
	\end{align}
	In particular, for $i=1$ we obtain $\partial_{j+1}\mathcal{E}^1=0$ for each $j\geq2$, that is $\partial_{l}\mathcal{E}^1=0$ for each $l\geq3$. By inductively assuming that $\partial_{l}\mathcal{E}^m=0$ for each $l\geq m+2$ for each $m\leq i-1$ for a fixed $i\leq n-1$, we deduce that $\partial_{l}\mathcal{E}^i=0$ for each $l\geq i+2$ by choosing $j=l-1$ in \eqref{AlmostAtlas_reduced}. This proves that $\partial_l\mathcal{E}^m=0$ for each $l\geq m+2$. By choosing $k=i$ and $j=2$ in \eqref{AlmostAtlas}, we get
	\begin{align}
		&\partial_{i+1}\mathcal{E}^i=\partial_2\mathcal{E}^{1}+\partial_i\mathcal{E}^{i-1},
		\label{AlmostAtlas_reduced_bis}
	\end{align}
	which gives
	\begin{align}
		&\partial_2\mathcal{E}^{1}=-\partial_n\mathcal{E}^{n-1}
		\label{AA_step2_1}
	\end{align}
	for $i=n$ and
	\begin{align}
		&\partial_2\mathcal{E}^{1}=\partial_{i+1}\mathcal{E}^i-\partial_i\mathcal{E}^{i-1}
		\label{AA_step2_2}
	\end{align}
	for $i\leq n-1$. From \eqref{AA_step2_2}, one can prove by induction that
	\begin{align}
		&\partial_{i+1}\mathcal{E}^i=i\,\partial_2\mathcal{E}^1,\qquad i=1,\dots,n-1.
		\label{AA_functionaltostep2}
	\end{align}
	In order to prove that $\partial_{l}\mathcal{E}^m=0$ when $l=m+1$ for higher values of $m$, it is then enough to show that $\partial_2\mathcal{E}^1=0$. By choosing $i=n-1$ in \eqref{AA_functionaltostep2}, we get
	\begin{align}
		&\partial_{n}\mathcal{E}^{n-1}=(n-1)\,\partial_2\mathcal{E}^1,
		\notag
	\end{align}
	which together with \eqref{AA_step2_1} implies $\partial_2\mathcal{E}^1=0$.	We then proved that $\partial_l\mathcal{E}^m=0$ for each $l\geq m+1$. We are now left with showing that
	\begin{equation}
		\partial_l\mathcal{E}^m=(l-1)\,\partial_2\mathcal{E}^{m-l+2}-(l-2)\,\partial_1\mathcal{E}^{m-l+1}\notag
	\end{equation}
	for each $l\leq m$. We again proceed by induction over $l$, starting from $l=m$.
	\begin{itemize}
		\item[$\bullet$] {\bf l=m} We need to prove
		\begin{equation}
			\partial_m\mathcal{E}^m=(m-1)\partial_2\mathcal{E}^{2}-(m-2)\partial_1\mathcal{E}^1\text{.}
			\label{diffmax}
		\end{equation}
		We proceed by induction over $m$.
		\begin{itemize}
			\item[$\bullet$] {\bf m=2} Condition \eqref{diffmax} trivially holds for $m=2$.
			\item[$\bullet$] {\bf m=h-1} Let us suppose $\partial_m\mathcal{E}^m=(m-1)\partial_2\mathcal{E}^{2}-(m-2)\partial_1\mathcal{E}^1$ for each ${m=2,\dots,h-1}$ (for some fixed $h\geq3$).
			\item[$\bullet$] {\bf m=h} By choosing $i=h$, $j=h-1$, $k=2$ in \eqref{AlmostAtlas}, we obtain
			\begin{align}
				&-\partial_{h}\mathcal{E}^{h}+\partial_{h-1}\mathcal{E}^{h-1}+\partial_{2}\mathcal{E}^{2}=\partial_{1}\mathcal{E}^{1}
				\notag
			\end{align}
			which by inductive hypothesis yields
			\begin{align}
				\partial_{h}\mathcal{E}^{h}&=\partial_{h-1}\mathcal{E}^{h-1}+\partial_{2}\mathcal{E}^{2}-\partial_{1}\mathcal{E}^{1}=(h-1)\partial_2\mathcal{E}^{2}-(h-2)\partial_1\mathcal{E}^1\text{.}
				\notag
			\end{align}
		\end{itemize}
		Condition \eqref{diffmax} is proved.
		\item[$\bullet$] {\bf l=t+1} Let us suppose
		\[
		\partial_l\mathcal{E}^m=(l-1)\,\partial_2\mathcal{E}^{m-l+2}-(l-2)\,\partial_1\mathcal{E}^{m-l+1}
		\]
		for every $l\geq t+1$, for some fixed $t\geq2$.
		\item[$\bullet$] {\bf l=t} We show
		\[
		\partial_t\mathcal{E}^m=(t-1)\,\partial_2\mathcal{E}^{m-t+2}-(t-2)\,\partial_1\mathcal{E}^{m-t+1}\text{.}
		\]
		By choosing $i=m+1$, $j=t$, $k=2$ in \eqref{AlmostAtlas}, we obtain 
		\begin{equation}
			\partial_{t+1}\mathcal{E}^{m+1}-\partial_t\mathcal{E}^{m}=\partial_2\mathcal{E}^{m-t+2}-\partial_1\mathcal{E}^{m-t+1}
			\label{diffshift}
		\end{equation}
		that yields
		\begin{equation}
			\partial_t\mathcal{E}^{m}=\partial_{t+1}\mathcal{E}^{m+1}-\partial_2\mathcal{E}^{m-t+2}+\partial_1\mathcal{E}^{m-t+1}=(t-1)\,\partial_2\mathcal{E}^{m-t+2}-(t-2)\,\partial_1\mathcal{E}^{m-t+1}
			\notag
		\end{equation}
		by the inductive hypothesis.
	\end{itemize}
\end{proof}                      

The above proposition leads to the following result.
\begin{theorem}
	Given a regular $F$-manifold with a single Jordan block with (generalized) canonical coordinates $u^1,\dots,u^n$ where the structure constants of the product and the components of the unit vector field respectively read $c^i_{jk}=\delta^i_{j+k-1}$ and $e^i=\delta^i_1$, an eventual identity must be of the form
	\[
	\mathcal{E}=\overset{n}{\underset{i=1}{\sum}}\mathcal{E}^{i}(u^{1},\dots,u^{i})\frac{\partial}{\partial u^{i}}
	\]
	where the functions $\{\mathcal{E}^{i}\}_{i=1,\dots,n}$ are solutions to \eqref{Atlas}.
\end{theorem}

It is straightforward to check the consistence of the system of equations for the components of $\mathcal{E}$ so \eqref{Atlas} gives a compatible system of PDEs, namely
	\begin{equation}
		\partial_i\partial_j\mathcal{E}^m=\partial_j\partial_i\mathcal{E}^m,\qquad\qquad i,j=1,\dots,n,
		\label{compatEJ=1}
	\end{equation}
for each $m=1,\dots,n$. 
	
Given the triangular nature of these equations, one may solve them term-by-term, and the first few terms are
	
\begin{align}
		\mathcal{E}^1&=f_1(u^1)\notag\\
		\mathcal{E}^2&=f_2(u^1,u^2)\notag\\
		\mathcal{E}^3&=(2\partial_2f_2-f_1')u^3+f_3(u^1,u^2)\notag\\
		\mathcal{E}^4&=(3\partial_2f_2-2f_1')u^4+2(\partial_2^2f_2)(u^3)^2+(2\partial_2f_3-\partial_1f_2)u^3+f_4(u^1,u^2)\notag\\
		\mathcal{E}^5&=(4\partial_2f_2-3f_1')u^5+(3\partial_2f_3-2\partial_1f_2)u^4+6(\partial_2^2f_2)u^3u^4\notag\\&+\frac{4}{3}(\partial_2^3f_2)(u^3)^3+\frac{1}{2}\big(4\partial_2^2f_3-4\partial_1\partial_2f_2+f_1''\big)(u^3)^2\notag\\&+(2\partial_2f_4-\partial_1f_3)u^3+f_5(u^1,u^2)\notag
	\end{align}
for some functions $\{f_i\}_{i=1,\dots,5}$ of the coordinates $u^1,u^2$. This structure persists, giving:

\begin{prop} 
An eventual identity is determined by one function $f_1$ of the coordinate $u^1$ and $n-1$ functions $\{f_i\}_{i=2,\dots,n}$ of the coordinates $u^1,u^2$. Moreover,
the components of an eventual identity have the following properties:
\begin{align}
		\mathcal{E}^1&=f_1(u^1)\,,\notag\\
		\mathcal{E}^2&=f_2(u^1,u^2)\,,\notag
\end{align}
where $f_1$ and $f_2$ are arbitrary functions and, for $i\geq 3$, 
\[
\mathcal{E}^i = f_i(u_1,u_2) + 
\left\{ \begin{array}{c}
{\rm polynomial~in~}u_3\,,\dots\,,u_i{\rm~with~coefficients}\cr{\rm~depending~on~the~functions~}f_1\,,\ldots\,,f_{i-1}\cr {\rm and~their~derivatives}
\end{array}\right\}.
\]

\end{prop}
\begin{proof}
	For each $m=1,\dots,n$, let us denote by $\mathcal{P}_{1,2}^{3,\dots,m}$ the set of polynomial functions in the variables $\{u^i\}_{i=3,\dots,m}$ with coefficients being functions of the coordinates $u^1,u^2$. We want to prove that for each $m=1,\dots,n$ we have
	\begin{align}
		\mathcal{E}^m=P^m(u^1,u^2;u^3,\dots,u^m)+f_m(u^1,u^2)
		\label{Epoly12}
	\end{align}
	for some function $f_m$ of $u^1,u^2$ and a function $P^m(u^1,u^2;u^3,\dots,u^m)\in\mathcal{P}_{1,2}^{3,\dots,m}$ which is uniquely determined up to $f_1,\dots,f_{m-1}$. The above example proves \eqref{Epoly12} for $m\leq5$. Let us assume \eqref{Epoly12} holds for $m\leq M-1$, for some fixed $M\leq n$, and show it holds for $m=M$. For each $l\geq3$ (without loss of generality $l\leq M$) we have
	\begin{align}
		\partial_l\mathcal{E}^M&\overset{\eqref{Atlas}}{=}(l-1)\,\partial_2\mathcal{E}^{M-l+2}-(l-2)\,\partial_1\mathcal{E}^{M-l+1}
		\notag
	\end{align}
	thus, since $M-l+2\leq M-1$, by induction we obtain 
	\begin{align}
		\partial_l\mathcal{E}^M&=(l-1)\,\partial_2P^{M-l+2}(u^1,u^2;u^3,\dots,u^{M-l+2})+(l-1)\,\partial_2f_{M-l+2}(u^1,u^2)\notag\\&-(l-2)\,\partial_1P^{M-l+1}(u^1,u^2;u^3,\dots,u^{M-l+1})-(l-2)\,\partial_1f_{M-l+1}(u^1,u^2)\text{.}
		\notag
	\end{align}
	In particular $\partial_l\mathcal{E}^M\in\mathcal{P}_{1,2}^{3,\dots,M-l+2}\subseteq\mathcal{P}_{1,2}^{3,\dots,M}$ for each $l\geq3$, thus
	\begin{align}
		\mathcal{E}^M=P^M(u^1,u^2;u^3,\dots,u^M)+f_M(u^1,u^2)
		\notag
	\end{align}
	for a function $f_M$ of $u^1,u^2$ and a function $P^M\in\mathcal{P}_{1,2}^{3,\dots,M}$ whose coefficents only depend on $f_1,\dots,f_{M-1}$.
\end{proof}

\medskip

The following example is not unexpected - but is included to show the use of the above result.

\begin{example}
	The Euler vector field
	\[
	E=\overset{n}{\underset{i=1}{\sum}}u^i\frac{\partial}{\partial u^{i}}
	\]
	is an eventual identity. In fact, given $\mathcal{E}^1(u^1)=u^1$ and $\mathcal{E}^2(u^2)=u^2$, for each ${i=3,\dots,n}$ condition \eqref{Atlas} gives
	\[
	\mathcal{E}^i(u^1,\dots,u^i)=u^i+f_i(u^1,\dots,u^{i-1})
	\]
	for some function $f_i(u^1,\dots,u^{i-1})$. In particular, for each $j=3,\dots,i-1$ we have
	\begin{align}
		\partial_jf_i&=\partial_j\mathcal{E}^i=(j-1)\partial_2\mathcal{E}^{i-j+2}-(j-2)\partial_1\mathcal{E}^{i-j+1}.
		\notag
	\end{align}
	Since $i-j+2\leq i-1$ (and a fortiori $i-j+1\leq i-1$) for each $j\geq3$, by induction we obtain 
	\begin{align}
		\partial_jf_i&=(j-1)\partial_2u^{i-j+2}-(j-2)\partial_1u^{i-j+1},\qquad j=3,\dots,i-1.
		\notag
	\end{align}
	Since $i-j+2\geq3$ and $i-j+1\geq2$ for each $j\leq i-1$, we obtain  $\partial_jf_i=0$ for each $j=3,\dots,i-1$, proving the function $f_i$ only depends on $u^1,u^2$. As expected, a first example of eventual identity is then provided by the Euler vector field (more generally, up to additive functions of $u^1$ and $u^2$).
\end{example}

\subsection{Vector fields preserving the dual structure}

The dual multiplication becomes quite complicated when expressed in the original basis,
\begin{align}
		\partial_i*\partial_j&=\overset{\sim}{c}^k_{ij}\,\partial_k,\qquad i,j,k=1,\dots,n,
		\notag
	\end{align}
where	
	\begin{align}
		\overset{\sim}{c}^i_{jk}&=(\mathcal{E}^{-1})^s c^i_{st}c^t_{jk}=(\mathcal{E}^{-1})^{i-j-k+2},\qquad i,j,k=1,\dots,n\text{.}
		\notag
	\end{align}
	
\noindent We now show that a basis may be found in which the dual multiplication $*$ has the same algebraic structure as the original $\circ$-multiplication.
\begin{lemma}
Let
\[
v_i=\mathcal{E}\circ \alpha^{i-1}\,,\qquad i=1\,,\ldots\,,n,
\]
for an arbitrary vector field $\alpha\,,$ and $\alpha^i=\alpha \circ \alpha^{i-1}$ where $\alpha^0=e\,.$ Then $v_1=\mathcal{E}$ is the unit for the dual multiplication and
\[
v_i * v_j = v_{i+j-1}\,.
\]
Moreover, if
\begin{equation}
\alpha=\sum_{i\geq 2} \alpha^i \frac{\partial~}{\partial u^i}
\label{alphaprop1}
\end{equation}
then $v_i=0$ for $i>n$ and $\{ v_i \}_{i=1\,,\ldots\,,n}$ forms a basis for the Jordan block.
\end{lemma}

\begin{proof}
The first part of the lemma is obvious. Note that since $\mathcal{E}$ is an eventual identity it is invertible, so the arbitrariness in $\alpha$ could be replaced by an arbitrariness in $v_2\,.$

From the triangular structure of the multiplication $\circ$ it follows, given (\ref{alphaprop1}), that
\[
v_i \in {\rm span}\left\{ \frac{\partial~}{\partial u^r}\,,r=i\,,\ldots\,,n
\right\}
\]
and hence that $v_i=0$ for $i>n$ and that $\{ v_i \}_{i=1\,,\ldots\,,n}$ forms a basis for the multiplication in the Jordan block.
\end{proof}

\subsection{Coordinates preserving the dual structure}

In this section we construct a dual coordinate system $w^1,\dots,w^n$ such that
\[
v_i=\frac{\partial}{\partial w^i},\qquad i=1,\dots,n\text{.}
\]
In general the vector fields $\{v_i\}_{i=1,\dots,n}$ do not commute, and hence to obtain this result one has to place constraints on the otherwise free vector $\alpha$ (or equivalently, on the vector $v_2$). In the semi-simple case this result is obvious, given the decomposition of the multiplication into one-dimensional blocks, and the new coordinate system is simply given by quadrature,
\[
w^i(u^i) = \int \frac{1}{\mathcal{E}^i(u^i)} du^i\,,\qquad i=1,\dots,n.
\]
For regular $F$-manifolds such a coordinate system will not, in general, exist. However, one may exploit various freedoms which do not occur in the semi-simple case to construct such a coordinate system. We begin with some general points.

While an eventual identity must, by definition, be invertible, the defining equation (\ref{Lie}) is well-defined without the requirement of invertibility. This motivates the following:

\begin{deft}
A vector field $\mathcal{E}$ satisfying equation (\ref{Lie}) is a weak-eventual identity. 
\end{deft}
\noindent 
So invertible weak-eventual identities are eventual identities. The reason for introducing this is that many of the results proved for eventual identities in \cite{DS} do not use the invertibility of vector field, and hence remain true for weak-eventual identities. For example, the product of weak-eventual identities is a weak-eventual identity. In particular, the proof of Proposition (\ref{formofE}) also does not use invertibility and hence is true for weak-eventual identities, and the following result from \cite{DS} also remains true for weak-eventual identities.
\begin{prop}
Let $\alpha$ be a weak-eventual identity, then the endomorphism $\alpha\circ$ is Nijenhuis, i.e. has vanishing Nijenhuis tensor.
\end{prop}

We now prove a general result that will be used to construct this coordinate system.

\begin{prop}
Let $\alpha\,,\beta_0\,,\beta_1$ be weak-eventual identities. Then
\[
[ \beta_0\circ\alpha^n, \beta_1\circ\alpha^m] = \alpha^{n+m-1} \circ\left\{ m\, [\beta_0,\beta_1\circ \alpha] + n\, [\beta_0\circ\alpha,\beta_1]\right\} - (n+m-1)\, \alpha^{m+n} \circ[\beta_0,\beta_1]\,.
\]
\end{prop}

\begin{proof}
Let $A_{n,m} = [ \beta_0\circ\alpha^n, \beta_1\circ\alpha^m]\,.$ Since $\alpha\circ$ is Nijenhuis,
\[
A_{n+1,m+1}=\alpha\circ A_{n,m+1} + \alpha\circ A_{n+1,m} - \alpha^2\circ A_{n,m}
\]
and hence, by induction, that
\[
A_{n,m}= \alpha^n\circ A_{0,m} + \alpha^m \circ A_{n,0} - \alpha^{n+m}\circ A_{0,0}\,.
\]
\noindent One may also show by induction that
\[
[\beta_0,\alpha^m] = (m-1) [e,\beta_0] \circ \alpha^m + m [\beta_0,\alpha] \circ\alpha^{m-1}\,,
\]
and with this, and again by induction, that
\[
[\beta_0,\beta_1\circ\alpha^m] = m\, [e,\beta_0]\circ\beta_1\circ\alpha^m + m \, [ \beta_0,\alpha] \circ \beta_1\circ \alpha^{m-1} + [\beta_0,\beta_1]\circ\alpha^m\,.
\]
From these, the result follows.

\end{proof}

As a corollary, if we set $\beta_0=\beta_1=\mathcal{E}$ and assume that $v_2$ is a weak-eventual identity (or equivalently, that $\alpha$ is a weak-eventual identity) then the above result gives
\[
[v_i,v_j] = (j-i) \, \alpha^{i+j-3} \circ [v_1,v_2]\,.
\]
Thus if we can construct a weak-eventual identity $v_2$ that commutes with $\mathcal{E}$ (recall $v_1=\mathcal{E}$) then all the $v_i$ commute and hence, by Frobenius' Theorem, there exists dual coordinate system $w^1,\dots,w^n$ such that
\[
v_i=\frac{\partial}{\partial w^i},\qquad i=1,\dots,n\text{.}
\]
Since we know the form of (weak)-eventual identities by Proposition (\ref{formofE}) we show that $[v_1,v_2]=0$ is a well-defined set of equations for the coefficients of $v_2$. Thus if  $v_2$ is weak-eventual identity its components must satisfy \eqref{Atlas}, that is
\begin{equation}
	\partial_lv_2^m=\begin{cases}
		(l-1)\,\partial_2v_2^{m-l+2}-(l-2)\,\partial_1v_2^{m-l+1}\qquad&\text{for }l\leq m\\
		0\qquad&\text{for }l>m
	\end{cases}
	\label{v2evId}
\end{equation}
or
\begin{equation}
	\partial_la_m=\begin{cases}
		(l-1)\,\partial_2a_{m-l+2}-(l-2)\,\partial_1a_{m-l+1}\qquad&\text{for }l\leq m\\
		0\qquad&\text{for }l>m
	\end{cases}
	\label{v2evId_a}
\end{equation}
for each $m=1,\dots,n$, where we have set $a_m=v^m_2$ for notational convenience. In particular, for each $m=2,\dots,n$ the function $a_m$ must only depend on the first $m$ coordinates. Moreover, we consider $a_2\neq0$.

\medskip

We now bring together various properties of the vector fields $\{v_i\}_{i=1,\dots,n}$.

\begin{prop}
\begin{itemize}
\item[(a)]  For all $j\geq3$ and $J\geq j \,,$
\begin{eqnarray}
		v_j^J& = & -\frac{1}{\mathcal{E}^1}\overset{J-j}{\underset{s=1}{\sum}}\,v_j^{J-s}\mathcal{E}^{s+1}+\frac{1}{\mathcal{E}^1}\overset{J-1}{\underset{s=1}{\sum}}\,v^s_{j-1}v^{J-s+1}_{2}\label{FFsim}\cr
	&=&\frac{a_2}{\mathcal{E}^1}v_{j-1}^{J-1}-\frac{1}{\mathcal{E}^1}\overset{J-j}{\underset{s=1}{\sum}}\,\bigg(v_j^{J-s}\mathcal{E}^{s+1}-v_{j-1}^{J-s-1}a_{s+2}\bigg)
		\label{FFhat}
\end{eqnarray}
\noindent and the vector fields $v_1,\dots,v_n$ are linearly independent.
\item[(b)] 	For each $j=1,\dots,n$ and $J\geq j$ the function $v_j^J$ only depends on the first $J-j+2$ coordinates.

\end{itemize}

\end{prop}

\begin{proof} The proof of part (a) is a straightforward calculation. From (\ref{FFhat}) it follows that, 
for each $j\geq3\,,$ 
	\begin{align}
		v_j^j&=\frac{a_2}{\mathcal{E}^1}v_{j-1}^{j-1}=\frac{(a_2)^2}{(\mathcal{E}^1)^2}v_{j-2}^{j-2}=\dots=\frac{(a_2)^{j-2}}{(\mathcal{E}^1)^{j-2}}v_{2}^{2}=\frac{(a_2)^{j-1}}{(\mathcal{E}^1)^{j-2}}\neq0
		\label{vjj}
	\end{align}
and hence that the vector fields $v_1,\dots,v_n$ are then linearly independent (assuming the leading term $a_2\neq 0$)\,.

For part (b), we show $\partial_kv_j^J=0$ for $k\geq J-j+3$. For $j=1$ and $j=2$, this trivially follows from \eqref{Atlas} and \eqref{v2evId} respectively. Let us now fix $j\geq3$. We proceed by induction over $J\geq j$.
	\begin{itemize}
		\item[$\bullet$] {\bf J=j} For each $k\geq3$ we have
		\begin{align}
			\partial_kv_j^j&\overset{\eqref{vjj}}{=}\partial_k\bigg(\frac{(a_2)^{j-1}}{(\mathcal{E}^1)^{j-2}}\bigg)=0\text{.}
			\notag
		\end{align}
		\item[$\bullet$] {\bf J-1} Let us suppose $\partial_kv_j^{\overline{J}}=0$ for each $k\geq\overline{J}-j+3$, for each $\overline{J}=j,\dots,J-1$ (for a fixed $J\geq j$).
		\item[$\bullet$] {\bf J} Let us fix $k\geq J-j+3$. We have
		\begin{align}
			\partial_kv_j^J&\overset{\eqref{FFhat}}{=}\frac{a_2}{\mathcal{E}^1}(\partial_kv_{j-1}^{J-1})-\frac{1}{\mathcal{E}^1}\overset{J-j}{\underset{s=1}{\sum}}\,\bigg((\partial_kv_j^{J-s})\mathcal{E}^{s+1}\notag\\&+v_j^{J-s}(\partial_k\mathcal{E}^{s+1})-(\partial_kv_{j-1}^{J-s-1})a_{s+2}-v_{j-1}^{J-s-1}(\partial_ka_{s+2})\bigg)
			\notag
		\end{align}
		where $\partial_kv_{j-1}^{J-1}$ and $\partial_kv_{j-1}^{J-s-1}$ vanish by induction over $j$, $\partial_kv_j^{J-s}$ vanishes by induction over $J$. By the steps $j=1$ and $j=2$, the quantities $\partial_k\mathcal{E}^{s+1}$ and $\partial_ka_{s+2}$ vanish as well for each $s\leq J-j$. Then $\partial_kv_j^J=0$.
	\end{itemize}
\end{proof}

Recall we require the vector field $v_2$ to have two properties: that it is a weak-eventual identity, and that it commutes with $\mathcal{E}\,.$
These two requirements impose conditions on the partial derivatives $\{\partial_la_m\}_{m=2,\dots,n}$ for $l\neq2$. More precisely, the condition for $v_2$ to be an eventual identity is given by \eqref{v2evId_a}\,, therefore for each ${m=3,\dots,n}$ it fixes $\{\partial_la_m\}_{l=3,\dots,m}$ in terms of $\{\partial_2a_j,\,\partial_1a_j\}_{j=2,\dots,m-1}$. 
The condition $[v_1,v_2]=0$ amounts to
	\begin{align}
		\partial_1a_m&=-\frac{1}{\mathcal{E}^1}\overset{m}{\underset{l=2}{\sum}}\big(\mathcal{E}^l\,\partial_la_m-a_l\,\partial_l\mathcal{E}^m\big)\notag\\&\overset{\eqref{Atlas}}{\underset{\eqref{v2evId_a}}{=}}-\frac{1}{\mathcal{E}^1}\overset{m}{\underset{l=2}{\sum}}\bigg((l-1)\big(\mathcal{E}^l\,\partial_2a_{m-l+2}-a_l\,\partial_2\mathcal{E}^{m-l+2}\big)\notag\\&-(l-2)\big(\mathcal{E}^l\,\partial_1a_{m-l+1}-a_l\,\partial_1\mathcal{E}^{m-l+1}\big)\bigg)
		\notag
	\end{align}
	thus for each $m=2,\dots,n$ it fixes the quantity $\partial_1a_m$ in terms of the quantities $\{\partial_2a_j,\,\partial_2\mathcal{E}^{j}\}_{j=2,\dots,m}$ and $\{\partial_1a_j,\,\partial_1\mathcal{E}^{j}\}_{j=1,\dots,m-1}$ where $a_1=0$.
	
	In order to prove that $v_2$ is both an eventual identity and commutes with $v_1=\mathcal{E}$, then, one must verify that for each $m=2,\dots,n$ the conditions \eqref{v2evId_a} and
	\begin{align}
		\partial_1a_m&=-\frac{1}{\mathcal{E}^1}\overset{m}{\underset{t=2}{\sum}}\big(\mathcal{E}^t\,\partial_ta_m-a_t\,\partial_t\mathcal{E}^m\big)
		\label{D1am}
	\end{align}
define a {\sl compatible} system of PDEs. Since we already checked the compatibility conditions for the system defining an eventual identity (\eqref{compatEJ=1} for \eqref{Atlas}, where \eqref{Atlas} rewrites as \eqref{v2evId_a} for the eventual identity $v_2$), we only need to prove that ${\partial_1\partial_la_m=\partial_l\partial_1a_m}$ for each $l=2,\dots,m$ (for $l>m$ the condition becomes trivial) for each $m=2,\dots,n$.
	
Consider the case when $m=2$. Since \eqref{v2evId_a} is vacuous, the only equation for $a_2$ is 
	\begin{align}
		\partial_1a_2&=-\frac{1}{\mathcal{E}^1}\big(\mathcal{E}^2\,\partial_2a_2-a_2\,\partial_2\mathcal{E}^2\big)
		\notag
	\end{align}
	that is
	\begin{align}
		\mathcal{E}^1\partial_1a_2+\mathcal{E}^2\,\partial_2a_2-a_2\,\partial_2\mathcal{E}^2=0.
		\label{Q}
	\end{align}
	Since this is the only equation for $a_2$, no compatibility condition is required. In other words, we automatically have
	\[
	\partial_1\partial_2a_2=\partial_2\partial_1a_2.
	\]
	We will now show how \eqref{Q} is the only non-trivial relation appearing among the compatibility conditions for the above system.
	
	Let us now consider $m=3$. We need to show that $\partial_1\partial_la_3=\partial_l\partial_1a_3$ for $l=2,3$. Let us start from $l=3$. We have
	\begin{align}
		\partial_1\partial_3a_3&=\partial_1(2\partial_2a_2)=2\partial_2\partial_1a_2=-\frac{2}{\mathcal{E}^1}\partial_2\big(\mathcal{E}^2\,\partial_2a_2-a_2\,\partial_2\mathcal{E}^2\big)=-\frac{2}{\mathcal{E}^1}\big(\mathcal{E}^2\,\partial^2_2a_2-a_2\,\partial^2_2\mathcal{E}^2\big)
		\notag
	\end{align}
	and
	\begin{align}
		\partial_3\partial_1a_3&=-\frac{1}{\mathcal{E}^1}\overset{3}{\underset{t=2}{\sum}}\partial_3\big(\mathcal{E}^t\,\partial_ta_3-a_t\,\partial_t\mathcal{E}^3\big)=-\frac{1}{\mathcal{E}^1}\big(\mathcal{E}^2\,\partial_3\partial_2a_3-a_2\,\partial_3\partial_2\mathcal{E}^3\big)\notag\\
		&=-\frac{2}{\mathcal{E}^1}\big(\mathcal{E}^2\,\partial^2_2a_2-a_2\,\partial^2_2\mathcal{E}^2\big)
		\notag
	\end{align}
	proving $\partial_1\partial_3a_3=\partial_3\partial_1a_3$. Let us now consider $l=2$. In order to prove the compatibility condition $\partial_1\partial_2a_3=\partial_2\partial_1a_3$, let us rewrite \eqref{v2evId_a} for $m=3$ (i.e. $\partial_3a_3=2\,\partial_2a_2$ for $a_3(u^1,u^2,u^3)$) as $a_3=2\,(\partial_2a_2)\,u^3+g_3(u^1,u^2)$ for some function $g_3$ of $u^1$, $u^2$. Without loss of generality, let $g_3$ vanish. We have
	\begin{align}
		\partial_1\partial_2a_3&=\partial_1(2\,(\partial^2_2a_2)\,u^3)=2\,(\partial_1\partial^2_2a_2)\,u^3
		\notag
	\end{align}
	and
	\begin{align}
		\partial_2\partial_1a_3=&-\frac{1}{\mathcal{E}^1}\overset{3}{\underset{t=2}{\sum}}\partial_2\big(\mathcal{E}^t\,\partial_ta_3-a_t\,\partial_t\mathcal{E}^3\big)\notag\\
		=&-\frac{1}{\mathcal{E}^1}\big(\partial_2\mathcal{E}^2\,\partial_2a_3+\mathcal{E}^2\,\partial^2_2a_3+\partial_2\mathcal{E}^3\,\partial_3a_3+\mathcal{E}^3\,\partial_2\partial_3a_3\notag\\
		&-\partial_2a_2\,\partial_2\mathcal{E}^3-a_2\,\partial^2_2\mathcal{E}^3-\partial_2a_3\,\partial_3\mathcal{E}^3-a_3\,\partial_2\partial_3\mathcal{E}^3\big)
		\notag
	\end{align}
	where $a_3=2\,(\partial_2a_2)\,u^3$ and $\mathcal{E}^3=(2\,\partial_2\mathcal{E}^2-\partial_1\mathcal{E}^1)\,u^3+f_3(u^1,u^2)$ for some function $f_3$ of $u^1$, $u^2$. Without loss of generality, let $f_3$ vanish. We obtain 
	\begin{align}
		\partial_2\partial_1a_3=&-\frac{2u^3}{\mathcal{E}^1}\big(\partial_2\mathcal{E}^2\,\partial_2^2a_2+\mathcal{E}^2\,\partial_2^3a_2-\partial_2^2\mathcal{E}^2\,\partial_2a_2-a_2\,\partial_2^3\mathcal{E}^2\big)
		\notag
	\end{align}
	which, by taking into account that applying $\partial_2$ twice to \eqref{Q} gives
	\begin{align}
		\mathcal{E}^1\partial_1\partial^2_2a_2+\partial_2\mathcal{E}^2\,\partial^2_2a_2+\mathcal{E}^2\,\partial^3_2a_2-\partial_2a_2\,\partial^2_2\mathcal{E}^2-a_2\,\partial^3_2\mathcal{E}^2=0,
		\label{QQQ}
	\end{align}
	becomes
	\begin{align}
		\partial_2\partial_1a_3=&2u^3\,\big(\partial_1\partial^2_2a_2\big)=\partial_1\partial_2a_3.
		\notag
	\end{align}
	
	It is possible to prove by induction over $m$ that $\partial_1\partial_la_m=\partial_l\partial_1a_m$ for $l=2,\dots,m$, starting from the above case of $m=3$. In fact, by fixing $m=3,\dots,n$ and assuming
	\begin{align}
		\partial_1\partial_la_k=\partial_l\partial_1a_k,\qquad l=2,\dots,k,
		\label{IndHpCompatm}
	\end{align}
	for each $k=3,\dots,m-1$, one can show that $\partial_1\partial_la_m=\partial_l\partial_1a_m$ for each $l=2,\dots,m$. This can be proved by using an inner induction over the difference $m-l$, starting from $m-l=0$. More precisely, starting from $l=m$ we have
	\begin{align}
		\partial_1\partial_ma_m=&\partial_1\big((m-1)\partial_2a_2\big)=(m-1)\partial_1\partial_2a_2=-\frac{m-1}{\mathcal{E}^1}\,\partial_2\big(\mathcal{E}^2\,\partial_2a_2-a_2\,\partial_2\mathcal{E}^2\big)\notag\\
		=&-\frac{m-1}{\mathcal{E}^1}\,\big(\mathcal{E}^2\,\partial^2_2a_2-a_2\,\partial^2_2\mathcal{E}^2\big)
		\notag
	\end{align}
	and
	\begin{align}
		\partial_m\partial_1a_m=&-\frac{1}{\mathcal{E}^1}\overset{m}{\underset{t=2}{\sum}}\partial_m\big(\mathcal{E}^t\,\partial_ta_m-a_t\,\partial_t\mathcal{E}^m\big)=-\frac{1}{\mathcal{E}^1}\overset{m}{\underset{t=2}{\sum}}\big(\mathcal{E}^t\,\partial_m\partial_ta_m-a_t\,\partial_m\partial_t\mathcal{E}^m\big)\notag\\
		=&-\frac{1}{\mathcal{E}^1}\overset{m}{\underset{t=2}{\sum}}\big(\mathcal{E}^t\,\partial_t((m-1)\partial_2a_2)-a_t\,\partial_t((m-1)\partial_2\mathcal{E}^2-(m-2)\partial_1\mathcal{E}^1)\big)\notag\\
		=&-\frac{m-1}{\mathcal{E}^1}\big(\mathcal{E}^2\,\partial^2_2a_2-a_2\,\partial^2_2\mathcal{E}^2\big)=\partial_1\partial_ma_m.
		\notag
	\end{align}
	By fixing $s=1,\dots,m-3$ and inductively assuming $\partial_1\partial_la_m=\partial_l\partial_1a_m$ for each $l=m-s+1,\dots,m$, one can show that $\partial_1\partial_la_m=\partial_l\partial_1a_m$ for $l=m-s$, implying $\partial_1\partial_la_m=\partial_l\partial_1a_m$ for each $l=3,\dots,m$.

	In order to prove that the same relation holds for $l=2$ as well, we consider the following result. For each $k=3,\dots,n$, let us denote by $\hat{\mathcal{P}}_k$ the set of polynomials in the variables $u^3,\dots,u^k$ with coefficients being polynomials in the derivatives (up to some positive integer order) of $a_2$ with respect to $u^1$, $u^2$.
	\begin{lemma}
		For each $m=3,\dots,n$ and $k=3,\dots,m$ the function $a_m$ can be written as
		\begin{align}
			a_m(u^1,\dots,u^m)=P^{(k)}_m(u^1,\dots,u^{m})+C^{(m-k+2)}_m(u^1,\dots,u^{m-k+2})
			\label{ampolyder_claim}
		\end{align}
		for some $P^{(k)}_m\in\hat{\mathcal{P}}_{m}$ and some function $C^{(m-k+2)}_m$ of $u^1,\dots,u^{m-k+2}$.
	\end{lemma}
	\begin{proof}
		For $m=3$, by \eqref{v2evId_a} we immediately have \eqref{ampolyder_claim} for
		\[
		P^{(3)}_3(u^1,u^2,u^3)=2\,\partial_2a_2\,u^3
		\]
		and some function $C^{(2)}_3$ of $u^1,u^2$. Let us fix $m=4,\dots,n$ and assume that for each $h=3,\dots,m-1$ and $k=3,\dots,h$ there exist some $P^{(k)}_h\in\hat{\mathcal{P}}_{h}$ and some function $C^{(h-k+2)}_h$ of $u^1,\dots,u^{h-k+2}$ such that	
		\begin{align}
			a_h(u^1,\dots,u^h)=P^{(k)}_h(u^1,\dots,u^{h})+C^{(h-k+2)}_h(u^1,\dots,u^{h-k+2}).
			\label{ampolyder_claim_indhp}
		\end{align}		
		We have to show \eqref{ampolyder_claim} for each $k=3,\dots,m$. We proceed by induction over $k$, starting from $k=3$. By \eqref{v2evId_a} we have
		\[
		\partial_ma_m=(m-1)\,\partial_2a_2
		\]
		thus
		\[
		a_m=(m-1)\,\partial_2a_2\,u^m+A_m(u^1,\dots,u^{m-1})
		\]
		for some function $A_m$ of $u^1,\dots,u^{m-1}$. This proves \eqref{ampolyder_claim} for $k=3$, with
		\[
		P^{(3)}_m(u^1,\dots,u^{m})=(m-1)\,\partial_2a_2\,u^m
		\]
		and
		\[
		C^{(m-1)}_m(u^1,\dots,u^{m-1})=A_m(u^1,\dots,u^{m-1}).
		\]
		Let us fix $k=4,\dots,m$ and assume that for each $h=3,\dots,k-1$ there exist some $P^{(h)}_m\in\hat{\mathcal{P}}_{m}$ and some function $C^{(m-h+2)}_m$ of $u^1,\dots,u^{m-h+2}$ such that	
		\begin{align}
			a_m(u^1,\dots,u^m)=P^{(h)}_m(u^1,\dots,u^{m})+C^{(m-h+2)}_m(u^1,\dots,u^{m-h+2}).
			\label{ampolyder_claim_indhp_bis}
		\end{align}
		By choosing $h=k-1$ in \eqref{ampolyder_claim_indhp_bis} we obtain 
		\begin{align}
			a_m(u^1,\dots,u^m)=P^{(k-1)}_m(u^1,\dots,u^{m})+C^{(m-k+3)}_m(u^1,\dots,u^{m-k+3})
			\label{ampolyder_claim_k-1}
		\end{align}
		for some $P^{(k-1)}_m\in\hat{\mathcal{P}}_{m}$ and some function $C^{(m-k+3)}_m$ of $u^1,\dots,u^{m-k+3}$. By combining \eqref{ampolyder_claim_k-1} with \eqref{v2evId_a}, we have
		\begin{align}
			\partial_{m-k+3}C^{(m-k+3)}_m&=\partial_{m-k+3}a_m-\partial_{m-k+3}P^{(k-1)}_m\notag\\
			&=(m-k+2)\partial_2a_{k-1}-(m-k+1)\partial_1a_{k-2}-\partial_{m-k+3}P^{(k-1)}_m\in\hat{\mathcal{P}}_m
			\notag
		\end{align}
		as $a_{k-1},a_{k-2},P^{(k-1)}_m\in\hat{\mathcal{P}}_m$ by inductive assumption%, since clearly $\hat{\mathcal{P}}_{j-1}\subseteq\hat{\mathcal{P}}_j$ for each $j$
		. Then
		\begin{align}
			C^{(m-k+3)}_m(u^1,\dots,u^{m-k+3})&=Q^{(m-k+3)}_m(u^1,\dots,u^m)+B_m^{(m-k+2)}(u^1,\dots,u^{m-k+2})
			\notag
		\end{align}
		for some $Q^{(m-k+3)}_m\in\hat{\mathcal{P}}_m$ and some function $B_m^{(m-k+2)}$ of $u^1,\dots,u^{m-k+2}$. Therefore we obtain 
		\begin{align}
			a_m(u^1,\dots,u^m)=P^{(k)}_m(u^1,\dots,u^{m})+C^{(m-k+2)}_m(u^1,\dots,u^{m-k+2})
			\notag
		\end{align}
		for
		\[
		P^{(k)}_m(u^1,\dots,u^{m})=P^{(k-1)}_m(u^1,\dots,u^{m})+Q^{(m-k+3)}_m(u^1,\dots,u^m)\in\hat{\mathcal{P}}_m
		\]
		and
		\[
		C^{(m-k+2)}_m(u^1,\dots,u^{m-k+2})=B_m^{(m-k+2)}(u^1,\dots,u^{m-k+2})
		\]
		proving \eqref{ampolyder_claim}.
	\end{proof}
	\bigskip\noindent
	By choosing $k=m$ in \eqref{ampolyder_claim}, for each $m=3,\dots,n$ we obtain 
	\begin{align}
		a_m(u^1,\dots,u^m)=P^{(m)}_m(u^1,\dots,u^{m})+C^{(2)}_m(u^1,u^{2}).
	\end{align}
	In particular, $a_m\in\hat{\mathcal{P}}_m$ and as a consequence
	\[
	\partial_1\partial_2a_m=\partial_2\partial_1a_m
	\]
	for each $m=3,\dots,n$.
%\end{rem}

Putting these results together gives the following.

\begin{theorem}
Let us consider a regular $F$-manifold consisting of one Jordan block of size $n$, endowed with an eventual identity $\mathcal{E}\,.$ Then there exists a weak-eventual identity $v_2$ that commutes with $\mathcal{E}$ and hence a family of commuting vector fields ${\{ v_1=\mathcal{E}\,,v_2\,,\ldots\,,v_n\}}$ and a dual regular $F$-manifold with multiplication
\[
v_i * v_j = v_{i+j-1},\qquad i,j=1,\dots,n,
\]
and a coordinate system $\{w^i\}_{i=1,\dots,n}$ such that
\[
v_i = \frac{\partial~}{\partial w^i}\,,\qquad i=1,\dots,n.
\]
\end{theorem}

\medskip

\section{The general case: multiple Jordan blocks}
Having proved various results for a single Jordan block we now consider the general case which multiple Jordan blocks, where $r\geq1$. As might be expected, the various blocks do not interact, though this is not a priori obvious. Condition \eqref{Lie} reads
\begin{align}
	&-\delta_{\beta\gamma}\partial_{(j+k-1)(\beta)}\mathcal{E}^{i(\alpha)}+\delta^\alpha_\gamma\partial_{j(\beta)}\mathcal{E}^{(i-k+1)(\alpha)}+\delta^\alpha_\beta\partial_{k(\gamma)}\mathcal{E}^{(i-j+1)(\alpha)}\notag\\&=\delta^\alpha_\beta\delta^\alpha_\gamma\overset{r}{\underset{\sigma=1}{\sum}}\partial_{1(\sigma)}\mathcal{E}^{(i-j-k+2)(\alpha)}
	\label{Lie_reg}
\end{align}
for each $\alpha,\beta,\gamma=1,\dots,r$ and for each $i=1,\dots,m_\alpha$, $j=1,\dots,m_\beta$, $k=1,\dots,m_\gamma$. If $\beta=\gamma$ in \eqref{Lie_reg} we obtain 
\begin{align}
	&-\partial_{(j+k-1)(\beta)}\mathcal{E}^{i(\alpha)}+\delta^\alpha_\beta\partial_{j(\beta)}\mathcal{E}^{(i-k+1)(\alpha)}+\delta^\alpha_\beta\partial_{k(\beta)}\mathcal{E}^{(i-j+1)(\alpha)}\notag\\&=\delta^\alpha_\beta\overset{r}{\underset{\sigma=1}{\sum}}\partial_{1(\sigma)}\mathcal{E}^{(i-j-k+2)(\alpha)}
	\notag
\end{align}
which yields
\begin{align}
	&\partial_{(j+k-1)(\beta)}\mathcal{E}^{i(\alpha)}=0
	\notag
\end{align}
for $\alpha\neq\beta$ (implying that $\mathcal{E}^{i(\alpha)}$ only depends on $u^{1(\alpha)},\dots,u^{m_\alpha(\alpha)}$) and
\begin{align}
	&-\partial_{(j+k-1)(\alpha)}\mathcal{E}^{i(\alpha)}+\partial_{j(\alpha)}\mathcal{E}^{(i-k+1)(\alpha)}+\partial_{k(\alpha)}\mathcal{E}^{(i-j+1)(\alpha)}=\partial_{1(\alpha)}\mathcal{E}^{(i-j-k+2)(\alpha)}
	\label{star}
\end{align}
for $\alpha=\beta$. If $\beta\neq\gamma$ in \eqref{Lie_reg} we obtain 
\begin{align}
	&\delta^\alpha_\gamma\partial_{j(\beta)}\mathcal{E}^{(i-k+1)(\alpha)}+\delta^\alpha_\beta\partial_{k(\gamma)}\mathcal{E}^{(i-j+1)(\alpha)}=0
	\notag
\end{align}
which trivially holds. An eventual identity must then be of the form
\[
\mathcal{E}=\overset{r}{\underset{\alpha=1}{\sum}}\overset{m_\alpha}{\underset{i=1}{\sum}}\mathcal{E}^{i(\alpha)}(u^{1(\alpha)},\dots,u^{m_\alpha(\alpha)})\frac{\partial}{\partial u^{i(\alpha)}}
\]
where the functions $\mathcal{E}^{i(\alpha)}(u^{1(\alpha)},\dots,u^{m_\alpha(\alpha)})$ are solutions to \eqref{star}. Since for each $\alpha=1,\dots,r$ condition \eqref{star} is analogous to \eqref{AlmostAtlas} and the functions $\{\mathcal{E}^{i(\alpha)}\}_{i=1,\dots,m_\alpha}$ only depend on the coordinates $u^1,\dots,u^{m_\alpha}$, the results about eventual identites from the previous section naturally extend to the general case where $r\geq1$.
\begin{theorem}
	Given a regular $F$-manifold with (generalized) canonical coordinates $\{u^{1(\alpha)},\dots,u^{m_\alpha(\alpha)}\}_{\alpha=1,\dots,r}$ where the structure constants of the product and the components of the unit vector field respectively read $c^{i(\alpha)}_{j(\beta)k(\gamma)}=\delta^\alpha_\beta\delta^\alpha_\gamma\delta^i_{j+k-1}$ and $e^{i(\alpha)}=\delta^i_1$, an eventual identity must be of the form
	\[
	\mathcal{E}=\overset{r}{\underset{\alpha=1}{\sum}}\overset{m_\alpha}{\underset{i=1}{\sum}}\mathcal{E}^{i(\alpha)}(u^{1(\alpha)},\dots,u^{m_\alpha(\alpha)})\frac{\partial}{\partial u^{i(\alpha)}}
	\]
	where for each $\alpha=1,\dots,r$ the functions $\{\mathcal{E}^{i(\alpha)}\}_{i=1,\dots,m_\alpha}$ are solutions to
	\begin{equation}
		\partial_{l(\alpha)}\mathcal{E}^{m(\alpha)}=\begin{cases}
			(l-1)\,\partial_{2(\alpha)}\mathcal{E}^{(m-l+2)(\alpha)}-(l-2)\,\partial_{1(\alpha)}\mathcal{E}^{(m-l+1)(\alpha)}\qquad&\text{for }l\leq m\\
			0\qquad&\text{for }l>m
		\end{cases}
		\label{Atlas_gen}
	\end{equation}
	for each $m=1,\dots,m_\alpha$.
\end{theorem}
\begin{rem}
	Condition \eqref{Atlas_gen} gives a compatible system of PDEs, namely
	\begin{equation}
		\partial_{i(\alpha)}\partial_{j(\beta)}\mathcal{E}^{m(\gamma)}=\partial_{j(\beta)}\partial_{i(\alpha)}\mathcal{E}^{m(\gamma)},\quad i=1,\dots,m_\alpha,\,j=1,\dots,m_\beta,\,\alpha,\beta=1,\dots,r
		\label{compatEJgen}
	\end{equation}
	for each $\gamma=1,\dots,r$ and $m=1,\dots,m_\gamma$. In fact, both the sides of such condition trivially vanish as soon as at least two of the indices $\alpha,\beta,\gamma$ are different (as each of the functions $\big\{\mathcal{E}^{i(\alpha)}\big\}_{i=1,\dots,m_\alpha,\,\alpha=1,\dots,r}$ only depends on the coordinates associated to the corresponding block). The only non-trivial case is then recovered when all of the Greek indices coincide, a case in which the proof of \eqref{compatEJ=1} can be adapted to prove \eqref{compatEJgen}.
\end{rem}
\begin{prop}
	Let us fix $\alpha=1,\dots,r$. For each $m=1,\dots,m_\alpha$ the $m(\alpha)$-th component of an eventual identity $\mathcal{E}$ is a polynomial function in the variables $\{u^{i(\alpha)}\}_{i=3,\dots,m_\alpha}$ with coefficients being functions of the coordinates $u^{1(\alpha)},u^{2(\alpha)}$. In particular, $\mathcal{E}^{1(\alpha)},\dots,\mathcal{E}^{m(\alpha)}$ only depend on a function $f_{1(\alpha)}$ of the coordinate $u^{1(\alpha)}$ and  $m-1$ functions $\{f_{i(\alpha)}\}_{i=2,\dots,m_\alpha}$ of the coordinates $u^{1(\alpha)},u^{2(\alpha)}$.
\end{prop}
\begin{proof}
	The argument for the case of a single Jordan block clearly extends to the more general case.
\end{proof}
\begin{rem}
	The above proposition implies that an eventual identity can be fully determined starting from $r$ functions $\big\{f_{1(\alpha)}(u^{1(\alpha)})\big\}_{\alpha=1,\dots,r}$ of one variable and $n-r$ functions $\big\{f_{i(\alpha)}(u^{1(\alpha)},u^{2(\alpha)})\big\}_{i=2,\dots,m_\alpha,\,\alpha=1,\dots,r}$ of two variables.
\end{rem}
\begin{example}
	The Euler vector field
	\[
	E=\overset{n}{\underset{i=1}{\sum}}u^i\frac{\partial}{\partial u^{i}}
	\]
	is an eventual identity for an arbitrary number and size of Jordan blocks. In fact, fixed $\alpha=1,\dots,r$ (without loss of generality, we consider $\alpha$ such that $m_\alpha\geq3$), given $\mathcal{E}^{1(\alpha)}(u^{1(\alpha)})=u^{1(\alpha)}$ and $\mathcal{E}^{2(\alpha)}(u^{2(\alpha)})=u^{2(\alpha)}$, for each $i=3,\dots,m_\alpha$ condition \eqref{Atlas_gen} gives
	\[
	\mathcal{E}^{i(\alpha)}(u^{1(\alpha)},\dots,u^{i(\alpha)})=u^{i(\alpha)}+f_{i(\alpha)}(u^{1(\alpha)},\dots,u^{(i-1)(\alpha)})
	\]
	for some function $f_{i(\alpha)}(u^{1(\alpha)},\dots,u^{(i-1)(\alpha)})$. Analogoulsy to the case of a single Jordan block, the function $f_{i(\alpha)}$ actually depends only on $u^{1(\alpha)},u^{2(\alpha)}$. As expected, a first example of eventual identity for arbitrary number of Jordan blocks is then provided by the Euler vector field.
\end{example}
\begin{rem}
	The structure constants of the dual product
	\[
	X*Y=\mathcal{E}^{-1}\circ X\circ Y,\qquad X\text{,}Y\in\mathfrak{X}(M),
	\]
	are
	\begin{align}
		\overset{\sim}{c}^{i(\alpha)}_{{j(\beta)}{k(\gamma)}}&=(\mathcal{E}^{-1})^{s(\sigma)} c^{i(\alpha)}_{s(\sigma)t(\tau)}c^{t(\tau)}_{{j(\beta)}{k(\gamma)}}=\delta^\alpha_\beta\delta^\alpha_\gamma(\mathcal{E}^{-1})^{(i-j-k+2)(\alpha)}
		\notag
	\end{align}
	for all suitable indices. The dual product is expressed on the coordinate vector fields as
	\begin{align}
		\partial_{i(\alpha)}*\partial_{j(\beta)}&=\overset{\sim}{c}^{k(\gamma)}_{{i(\alpha)}{j(\beta)}}\partial_{k(\gamma)}=\delta_{\alpha\beta}\overset{m_\alpha}{\underset{k=i+j-1}{\sum}}(\mathcal{E}^{-1})^{(k-i-j+2)(\alpha)}\partial_{k(\alpha)}\text{.}
		\notag
	\end{align}
	In particular, $\partial_{i(\alpha)}*\partial_{j(\beta)}=0$ for $\alpha\neq\beta$ and $\partial_{i(\alpha)}*\partial_{j(\alpha)}=0$ for $i+j\geq m_\alpha+2$.
\end{rem}
Analogously to the case of one block discussed in the previous section, we introduce vector fields $\{v_{1(\alpha)},\dots,v_{m_\alpha(\alpha)}\}_{\alpha=1,\dots,r}$ such that
\begin{equation}
	v_{i(\alpha)}*v_{j(\beta)}=\delta_{\alpha\beta}\,v_{(i+j-1)(\alpha)}\,\mathds{1}_{\{i+j\leq m_\alpha+1\}}\text{.}
	\notag
\end{equation}
\begin{thm}
	By setting
	\begin{itemize}
		\item[$\bullet$] $v_{1(\alpha)}=\overset{m_\alpha}{\underset{i=1}{\sum}}\,\mathcal{E}^{i(\alpha)}\partial_{i(\alpha)}$
		\item[$\bullet$] $v_{2(\alpha)}=\overset{m_\alpha}{\underset{i=2}{\sum}}\,a_{i(\alpha)}\,\partial_{i(\alpha)}$ for some functions $a_{2(\alpha)}\neq0,a_{3(\alpha)},\dots,a_{m_\alpha(\alpha)}$
		\item[$\bullet$] $v_{(i+1)(\alpha)}=\overset{m_\alpha}{\underset{k=i+1}{\sum}}\,(v_{(i+1)(\alpha)})^{k(\alpha)}\,\partial_{k(\alpha)}$ for $i\geq2$ with
		\begin{equation}
			(v_{(i+1)(\alpha)})^{k(\alpha)}=\overset{k-1}{\underset{a=i}{\sum}}\,\overset{k-a+1}{\underset{b=2}{\sum}}\,(v_{i(\alpha)})^{a(\alpha)}(v_{2(\alpha)})^{b(\alpha)}({\mathcal{E}}^{-1})^{(k-a-b+2)(\alpha)},\qquad k\geq i+1
			\notag
		\end{equation}
	\end{itemize}
	for each $\alpha=1,\dots,r$, we have
	\begin{equation}
		v_{i(\alpha)}*v_{j(\beta)}=\delta_{\alpha\beta}\,v_{(i+j-1)(\alpha)}\,\mathds{1}_{\{i+j\leq m_\alpha+1\}}
		\notag
	\end{equation}
	for each $\alpha,\beta=1,\dots,r$.
\end{thm}
\begin{proof}
	The proof for the case $r=1$ can be adapted to the general case where $r\geq1$, as for each $\alpha=1,\dots,r$ the functions $\{\mathcal{E}^{i(\alpha)}\}_{i=1,\dots,m_\alpha}$ and $\{a_{i(\alpha)}\}_{i=2,\dots,m_\alpha}$ only depend on the coordinates $u^{1(\alpha)},\dots,u^{m_\alpha(\alpha)}$.
\end{proof}

\medskip

In particular, we assume $v_{2(\alpha)}$ to be solution to \eqref{Atlas_gen} for each $\alpha=1,\dots,r$. In other words, we are assuming
\begin{equation}
	\overset{r}{\underset{\alpha=1}{\sum}}\overset{m_\alpha}{\underset{i=2}{\sum}}a_{i(\alpha)}\frac{\partial~}{\partial u^{i(\alpha)}}
	\notag
\end{equation}
to be a weak eventual identity.
\begin{rem}
	Analogous formulas to \eqref{FFsim} apply here as well. For each $\alpha=1,\dots,r$ we have
	\begin{equation}
		v_{j(\alpha)}^{J(\alpha)}=-\frac{1}{\mathcal{E}^{1(\alpha)}}\overset{J-j}{\underset{s=1}{\sum}}\,v_{j(\alpha)}^{(J-s)(\alpha)}\mathcal{E}^{(s+1)(\alpha)}+\frac{1}{\mathcal{E}^{1(\alpha)}}\overset{J-1}{\underset{s=1}{\sum}}\,v^{s(\alpha)}_{(j-1)(\alpha)}v^{(J-s+1)(\alpha)}_{2(\alpha)}
		\label{FFsim_gen}
	\end{equation}
	for each $j\geq3$ and $J\geq j$ and
	\begin{align}
		v_{j(\alpha)}^{J(\alpha)}&=\frac{a_{2(\alpha)}}{\mathcal{E}^{1(\alpha)}}v_{(j-1)(\alpha)}^{(J-1)(\alpha)}-\frac{1}{\mathcal{E}^{1(\alpha)}}\overset{J-j}{\underset{s=1}{\sum}}\,\bigg(v_{j(\alpha)}^{(J-s)(\alpha)}\mathcal{E}^{(s+1)(\alpha)}-v_{(j-1)(\alpha)}^{(J-s-1)(\alpha)}a_{(s+2)(\alpha)}\bigg)
		\label{FFhat_gen}
	\end{align}
	for each $j\geq3$ and $J\geq j$. In particular, for each $j\geq3$ we have
	\begin{align}
		v_{j(\alpha)}^{j(\alpha)}&=\frac{a_{2(\alpha)}}{\mathcal{E}^{1(\alpha)}}v_{(j-1)(\alpha)}^{(j-1)(\alpha)}=\frac{(a_{2(\alpha)})^{(j-1)(\alpha)}}{(\mathcal{E}^{1(\alpha)})^{(j-2)(\alpha)}}\neq0\text{.}
		\label{vjj_gen}
	\end{align}
	The vector fields $v_1,\dots,v_n$ are then linearly independent.
\end{rem}
The same considerations about the independence on coordinates which correspond to different blocks lead to the following.
\begin{prop}
	For each $\alpha=1,\dots,r$ the vector fields $\{v_{i(\alpha)}\}_{i\geq3}$ are solutions to \eqref{Atlas_gen}, namely
	\begin{equation}
		\overset{r}{\underset{\alpha=1}{\sum}}\overset{m_\alpha}{\underset{I=i}{\sum}}v_{i(\alpha)}^{I(\alpha)}\frac{\partial~}{\partial u^{I(\alpha)}}
		\notag
	\end{equation}
	is a weak eventual identity for each $i\geq3$. In particular, for each $j=1,\dots,m_\alpha$ and $J\geq j$ the function $v_{j(\alpha)}^{J(\alpha)}$ only depends on the coordinates $u^{1(\alpha)},\dots,u^{(J-j+2)(\alpha)}$.	
\end{prop}
Let us assume $[v_{1(\alpha)},v_ {2(\alpha)}]=0$ for each $\alpha=1,\dots,r$, that is for $m=2,\dots,m_\alpha$
\begin{align}
	\partial_1a_{m(\alpha)}&=-\frac{1}{\mathcal{E}^{1(\alpha)}}\overset{m}{\underset{l=2}{\sum}}\big(\mathcal{E}^{l(\alpha)}\,\partial_{l(\alpha)}a_{m(\alpha)}-a_{l(\alpha)}\,\partial_{l(\alpha)}\mathcal{E}^{m(\alpha)}\big)\text{.}
	\label{comm12_gen}
\end{align}
Analogously to the case of a single Jordan block, the request for $v_2$ to both be an eventual identity and to commute with $v_1=\mathcal{E}$ gives rise to a compatible system of PDEs for the components of $v_2$.
\begin{thm}
	The vector fields $\{v_{i(\alpha)}\}_{i=1,\dots,m_\alpha,\,\alpha=1,\dots,r}$ pairwise commute.
\end{thm}
\begin{proof}
	For each $\alpha=1,\dots,r$, the results about the commutativity of the vector fields $\{v_{i(\alpha)}\}_{i=1,\dots,m_\alpha}$ extend as well, thus the vector fields $\{v_{i(\alpha)}\}_{i=1,\dots,m_\alpha}$ pairwise commute. Since $[v_{i(\alpha)},v_{j(\beta)}]$ trivially vanishes for $\alpha\neq\beta$, the result is proved.
\end{proof}
\begin{cor}
	There exist coordinates $\{w^{1(\alpha)},\dots,w^{m_\alpha(\alpha)}\}_{\alpha=1,\dots,r}$ such that
	\[
	v_{i(\alpha)}=\frac{\partial}{\partial w^{i(\alpha)}}
	\]
	for each $\alpha=1,\dots,r$ and $i=1,\dots,m_\alpha$.
\end{cor}

\section{Comments}

The results of this paper may be extended in a number of different directions. There has been a renewed study of Nijenhuis operators, where we recall that an operator $\mathcal{L}$ is called a \emph{Nijenhuis operator} if its Nijenhuis torsion
	\begin{align}
		N_\mathcal{L}:\mathfrak{X}(M)\times\mathfrak{X}(M)&\to\mathfrak{X}(M)\notag\\
		(X,Y)&\mapsto N_\mathcal{L}(X,Y):=[\mathcal{L}X,\mathcal{L}Y]-\mathcal{L}[X,\mathcal{L}Y]-\mathcal{L}[\mathcal{L}X,Y]+\mathcal{L}^2[X,Y]
		\notag
	\end{align}
vanishes (see, for example, \cite{ALimrn,BKM,LP2}). It was proved in \cite{DS} that for {\sl any} $F$-manifold with eventual identity $\mathcal{E}$ the endomorphism $\mathcal{E}\circ$ is an Nijenhuis operator. In the semi-simple case this is just a diagonal endomorphism, but for the eventual identities constructed in this paper, the Nijenhuis operator takes the block-diagonal form with each block of the form

\begin{equation}
	\mathcal{L}_\alpha=\begin{bmatrix}\mathcal{E}^{1(\alpha)}&0&0&\dots&0&0\\\mathcal{E}^{2(\alpha)}&\mathcal{E}^{1(\alpha)}&0&\dots&0&0\\%u^3&u^2&u^1&\dots&0&0\\
		\vdots&\vdots&\vdots&\ddots&\vdots&\vdots\\\mathcal{E}^{(m_\alpha-1)(\alpha)}&\mathcal{E}^{(m_\alpha-2)(\alpha)}&\mathcal{E}^{(m_\alpha-3)(\alpha)}&\dots&\mathcal{E}^{1(\alpha)}&0\\\mathcal{E}^{m_\alpha(\alpha)}&\mathcal{E}^{(m_\alpha-1)(\alpha)}&\mathcal{E}^{(m_\alpha-2)(\alpha)}&\dots&\mathcal{E}^{2(\alpha)}&\mathcal{E}^{1(\alpha)}
	\end{bmatrix}.\notag
\end{equation}
and using the functional freedom in the choice of eventual identity one may construct large classes of examples, and it would be of interest to see how these fall into the classification in \cite{BKM}.

Non-semisimple Frobenius manifolds with an underlying regular $F$-manifold structure may also be constructed \cite{LP1}. Finally, the analogues of the first and second structural connections of a Frobenius manifolds were studied for $F$-manifolds with an eventual identity in \cite{DS2} (see also \cite{DH19}). It would be of interest to specialise this general construction to the case of a regular $F$-manifold, deriving analogues of the Darboux-Egorov equations. These would form an interesting multi-dimensional integrable system.

\bibliographystyle{alpha}

	SARA PERLETTI: Department of Mathematics and Applications, University of Milano-Bicocca, Milano 20125, Italy; E-mail address: s.perletti1@campus.unimib.it
	
\medskip
	
	IAN A. B. STRACHAN: School of Mathematics and Statistics, University of Glasgow, Glasgow G12 8TA, UK; E-mail address: ian.strachan@glasgow.ac.uk

\end{document}